\documentclass[11pt]{amsart}

\usepackage[dvipdfm,bookmarks=true]{hyperref}
\hypersetup{colorlinks,citecolor=blue,linkcolor=blue}

\usepackage{mathptmx}
\usepackage{amsmath}
\usepackage{amssymb}
\usepackage{amsthm}
\usepackage{array}
\usepackage[all,tips]{xy}
\usepackage{enumerate}
\usepackage{graphicx}
\usepackage{lmodern}
\usepackage{mathrsfs}

\newtheorem{theorem}{Theorem}
\newtheorem{lemma}[theorem]{Lemma}

\newtheorem{conjecture}[theorem]{Conjecture}
\theoremstyle{definition}

\theoremstyle{remark}
\newtheorem{rmk}[theorem]{Remark}

\DeclareMathOperator{\GL}{GL}

\DeclareMathOperator{\Spin}{Spin}

\newcommand{\Z}{\mathbb Z}
\newcommand{\Q}{\mathbb Q}

\newcommand{\F}{\mathbb F}
\renewcommand{\P}{\mathbf P}
\newcommand{\f}{\mathfrak f}

\newcommand{\V}{V}
\newcommand{\Vf}{\V_{\mathrm{f}}}

\newcommand{\tB}{\mathrm{B}}

\newcommand{\tD}{\mathrm{D}}

\newcommand{\Hy}{\mathbf H}
\newcommand{\Isom}{\mathrm{Isom}}

\newcommand{\G}{\mathrm G}

\newcommand{\GG}{\mathbf G}
\newcommand{\MM}{\mathbf M}

\renewcommand{\O}{\mathcal O}
\newcommand{\PP}{\mathbf P}

\newcommand{\bs}{\backslash}
\newcommand{\vol}{\mathrm{vol}}

\newcommand{\blambda}{\overline{\lambda}}

\begin{document}

\title{Hyperbolic manifolds of small volume}

\author{Mikhail Belolipetsky}
\thanks{Belolipetsky supported by a CNPq research grant}
\address{
IMPA\\
Estrada Dona Castorina, 110\\
22460-320 Rio de Janeiro, Brazil}
\email{mbel@impa.br}

\author{Vincent Emery}
\thanks{Emery supported by SNSF, project no. {\tt PA00P2-139672} and {\tt PZ00P2-148100}}
\address{
FSB-MATHGEOM\\
EPF Lausanne\\
B\^atiment MA, Station 8\\
CH-1015 Lausanne\\
Switzerland
}
\email{vincent.emery@gmail.com}

\date{\today}

%\subjclass{22E40 (primary); 11E57, 20G30, 51M25 (secondary)}
% Group Theory;  Geometric Topology, Number Theory

\begin{abstract}
We conjecture that for every dimension $n \neq 3$ there exists a noncompact hyperbolic $n$-manifold whose volume is smaller than the volume of any compact hyperbolic $n$-manifold. For dimensions $n \le 4$ and $n = 6$ this conjecture follows from the known results. In this paper we show that the conjecture is true for arithmetic hyperbolic $n$-manifolds of dimension $n\ge 30$.
\end{abstract}

\maketitle

\section{Introduction}
\label{intro}

By \emph{hyperbolic $n$-manifold} we mean a complete orientable manifold
that is locally isometric to the hyperbolic $n$-space $\Hy^n$.
Let $M$ be a complete noncompact hyperbolic $3$-manifold of finite volume. By Thurston's Dehn
surgery theorem there exists infinitely many compact hyperbolic $3$-manifolds
obtained from $M$ by Dehn filling \cite[Section~5.8]{Thur80}. These manifolds
all have their volume smaller than $\vol(M)$ \cite[Theorem~6.5.6]{Thur80}. It is 
known that the Dehn filling procedure is specific to manifolds of dimension $n = 3$
(cf. \cite[Section~2.1]{Anderson06}). We believe that so is the minimality of 
volume of compact manifolds and that for hyperbolic manifolds in other dimensions 
the following conjecture should hold:

\begin{conjecture}
 Let $N$ be a compact hyperbolic manifold of dimension $n\neq 3$. Then
 there exists a noncompact hyperbolic $n$-manifold $M$ whose volume is smaller
 than the volume of $N$.
 \label{conj}
\end{conjecture}
The set of volumes of hyperbolic $n$-manifolds is well-ordered (indeed, discrete if $n\neq 3$),
thus the conjecture states that the smallest complete hyperbolic $n$-manifold  is noncompact 
for $n \neq 3$. Conjecture \ref{conj} is known to be true for dimensions $n=2$, $4$ and $6$: 
for these $n$ there exist noncompact hyperbolic $n$-manifolds $M$ with
$|\chi(M)| = 1$ \cite{RatTsch00, ERT10}, whereas it is a general fact that the Euler
characteristic of a compact hyperbolic manifold is even
(cf. \cite[Theorem~4.4]{Ratcliffe2001899}).

In this paper we prove the conjecture for \emph{arithmetic} hyperbolic $n$-manifolds of
sufficiently large dimension:
\begin{theorem}
  Conjecture \ref{conj} holds for arithmetic hyperbolic $n$-manifolds of dimension $n \ge 30$.
  \label{thm:conj-ok}
\end{theorem}
A folklore conjecture states that the smallest volume is always attained on an arithmetic
hyperbolic $n$-manifold (in both compact and noncompact cases). We will refer to this statement as the 
\emph{minimal volume conjecture}. This conjecture is obvious for $n = 2$, it was
proved only recently for $n = 3$ \cite{GMM10}, and there is currently no any potential approach to 
resolving this conjecture in higher dimensions. If true,
the minimal volume conjecture together with Theorem~\ref{thm:conj-ok} would imply Conjecture~\ref{conj}
for dimensions $n\ge 30$. However, Conjecture~\ref{conj} is weaker than the minimal volume and we think that
it might be easier to attack it directly.

At this point we would like to mention the related picture for hyperbolic $n$-orbifolds. By \cite{Belo04,BelEme}, 
we know the compact and noncompact minimal volume arithmetic hyperbolic $n$-orbifolds in all dimensions $n\ge 4$. This is 
complemented 
by the previous results for $n=2$, $3$ (see [loc.cit.]). It follows that the smallest arithmetic hyperbolic 
$n$-orbifold is compact for $n = 2$, $3$, $4$ and noncompact for $n\ge 5$. For dimensions $n\le 9$ we do know the 
smallest volume noncompact hyperbolic $n$-orbifolds thanks to the result of Hild's thesis (\cite{Hild07}). It 
turns out that these orbifolds are arithmetic, supporting the orbifold version of the minimal volume conjecture.

Let us now describe the contents of the paper. The proof of the main theorem is based on
arithmetic techniques. We start with the minimal volume arithmetic
$n$-orbifold constructed in \cite{Belo04,BelEme}, which is noncompact
for $n\ge 5$. Then we construct a manifold cover of it and try to show
that the volume of this manifold is still smaller than the smallest compact arithmetic
$n$-orbifold, which is also known from our previous work. 
%Here comes the
%first challenge: The degree of the manifold covers grows rapidly with
%the dimension and a direct approach via taking torsion-free congruence
%kernels does not allow us to obtain any reasonable bounds for the
%dimension.
Constructing a manifold cover of an orbifold $\Gamma \bs \Hy^n$ is
equivalent to finding a torsion-free subgroup of $\Gamma$. 
Our method here is based on Lemma~\ref{lemma:p-torsion},
which can be thought of as a variant of Minkowski's lemma that asserts
that the kernel of $\GL_N(\Z) \to \GL_N(\Z/m)$ is torsion-free for $m>2$.
Minkowski's lemma is the classical tool to construct hyperbolic
manifolds from arithmetic subgroups (see for instance \cite[Section
4]{Ratcliffe2001899}). We observe that applying
Lemma~\ref{lemma:p-torsion} on two small primes (see Section
\ref{sec:method-details}) gives better results. In particular we show 
in Section~\ref{noncompact-manifold}  how our method applies to the 
problem and proves Theorem~\ref{thm:conj-ok} for dimensions $n\ge 33$.
On the other hand, Minkowski's lemma (with $m = 3$) cannot be used to obtain a direct
proof for $n < 50$. 

%torsion in congruence subgroups, and an important observation that
%taking the intersection of two subgroups that correspond to different
%small primes gives a torsion-free subgroup whose index is smaller than
%the index of a torsion-free subgroup that corresponds to a single
%sufficiently large prime. The details of the method are explained in
%Section~\ref{method-constr}. In Section~\ref{noncompact-manifold} we
%show how the method applies to the problem and prove
%Theorem~\ref{thm:conj-ok} for dimensions $n\ge 33$.

The dimension bound in Theorem \ref{thm:conj-ok} can be further improved because the compact
orbifolds of small volume considered above do have singularities and
hence the volumes of the smallest compact manifolds are larger. In order
to obtain better estimates for these volumes we use the arithmetic
information encoded in the Euler characteristic. This approach works
well for even dimensions (see Sections~\ref{sec:n=32} and
\ref{sec:n=30}). In odd dimensions the Euler characteristic is zero, but
we can still use a similar method. In order to do so we consider even
dimensional totally geodesic suborbifolds of the compact odd dimensional
orbifolds of small volume. The details of the argument are explained in
Section~\ref{sec:odd-dim-trick} and its application to the case $n=31$
is in Section~\ref{sec:n=31}. It is possible to use similar
considerations for smaller dimensions but the number of variants and the
computational difficulty increase rapidly for smaller $n$. The problem
becomes very difficult for $n<10$ and we expect that some new ideas or
significant computational advances are required in order to extend our
result to this range.

\medskip

\noindent\textbf{Acknowledgments.} Part of this work was done while the second author was visiting IMPA (Rio de Janeiro, Brazil). He wishes to thank the institute for the hospitality and support.

\section{A method of construction}
\label{method-constr}

\subsection{}
\label{sec:defns}

Let $\GG$ be a semisimple algebraic group defined over a number field
$k$. We denote by $\Vf$ the set of finite places of $k$. For each $v \in
\Vf$, we denote by $k_v$, $\O_v$ and $\f_v$ the completion, local ring
and residue field defined by $v$. We suppose that
$\GG$ is simply connected as an algebraic group.
By definition, a \emph{principal arithmetic subgroup} of $\GG(k)$
is a subgroup $\Lambda_P = \GG(k) \cap \prod_{v \in \Vf}
P_v$, where $P = (P_v)_{v\in\Vf}$ is a coherent collection of parahoric
subgroups $P_v \subset \GG(k_v)$ (see \cite{BorPra89}). By Bruhat-Tits
theory (see \cite{Tits79}), for each $v \in \Vf$, there exists a smooth
connected $\O_v$-group scheme $\PP_v$ with generic fiber $\GG$
and such that  $\PP_v(\O_v) = P_v$.
We denote by $\MM_v$ the $\f_v$-group defined as the  maximal
reductive quotient of $\PP_v / \f_v$.
The reduction map $P_v \to \MM_v(\f_v)$ is surjective. Moreover,
since $\GG$ is simply connected, we have that  $\PP_v/\f_v$ (and thus
$\MM_v$ as well) is connected.

\subsection{}
\label{sec:lemma}

Let $p$ be the rational prime above $v$, i.e., $p$ is the characteristic
of the residue field $\f_v$.
Let us denote by $P_v^* \subset P_v$ the pre-image under the reduction map of a
$p$-Sylow subgroup of $\MM_v(\f_v)$.

\begin{lemma}
  Each torsion element in $P_v^*$ has an order of the form $p^s$.
  \label{lemma:p-torsion}
\end{lemma}
\begin{proof}
  Let $K_v$ be the kernel of the map $P_v \to \PP_v(\f_v)$.  Let $g \in P_v^*$
  be an element of finite order $q$, and denote by $\overline{g}$ its
  image in $\PP_v(\f_v)$. It suffices to consider the case where $q$ is
  prime. The image of $P^*_v$ in $\PP_v(\f_v)$ is, by definition, the
  extension of a $p$-group by a unipotent $\f_v$-group. It follows that
  the image of $\P^*_v$ is itself a $p$-group, 
  and thus we 
  have either $\overline{g} = 1$, or $q=p$. But in the
  first case $g$ is contained in $K_v$, which is a pro-$p$ group (see
  \cite[Lemma 3.8]{PlaRap94}), and the conclusion follows as well.
\end{proof}

\subsection{}
\label{sec:method-details}

By choosing two primes $v$ and $w$ above two distinct rational
primes, and replacing $P_v$ (resp. $P_w$) by $P_v^*$ (resp. $P_w^*$)
in the coherent collection $P$, we obtain by Lemma \ref{lemma:p-torsion}
a torsion-free
subgroup $\Lambda^{v,w}_P \subset \Lambda_P$. By strong approximation of
$\GG$,
the index $[\Lambda_P : \Lambda^{v,w}_P]$ equals $[P_v: P^*_v] \cdot
[P_w:P_w^*]$.

By construction, the index $[P_v:P_v^*]$ is equal to the order of
$\MM_v(\f_v)$ divided by its highest $p$-factor. We will usually work with
parahoric subgroups for which $\MM_v$ is semisimple (for
instance, $P_v$ hyperspecial), and in this case
$[P_v:P_v^*]$ can be easily computed from the type of $\MM_v$ using the
list of orders of simple groups over a finite field (see \cite{Ono66}).

For example, if $\GG / k_v$ is a split group of type $\tB_r$ and  $P_v \subset
\GG(k_v)$ is hyperspecial, then $\MM_v$ is semisimple of the same type
and we obtain:
\begin{align}
  [P_v:P_v^*] &= \prod_{j=1}^r (q_v^{2j} - 1),
  \label{eq:index-Br}
\end{align}
where $q_v$ denotes the cardinality of the finite field $\f_v$. The product of
two such expressions gives us the index of the torsion-free subgroup
$\Lambda^{v,w}_P$ in $\Lambda_P$. If $\Lambda_P$ acts on $\Hy^n$ we can then
compute the volume of the quotient manifold $\Lambda^{v,w}_P\bs\Hy^n$ as a product of this
index and the covolume of $\Lambda_P$.

\section{Construction of noncompact manifolds}
\label{noncompact-manifold}

We will construct noncompact hyperbolic manifolds starting from the
arithmetic lattices of minimal covolume considered in \cite{Belo04,BelEme}.
They are best understood as
normalizers of principal arithmetic subgroups of $G = \Spin(n,1)$, the
latter Lie group being a two-fold  covering of $\Isom^+(\Hy^n)$.

\subsection{}

We showed in \cite[Section 2.1]{BelEme} that if a lattice $\Gamma \subset G$
contains
the center $Z$ of $\Spin(n,1)$, then the hyperbolic volume $\vol(\Gamma\bs
\Hy^n)$ is given by $2 \vol(S^n) \mu(\Gamma \bs G)$, where $S^n$ is the
$n$-sphere of constant curvature $1$ and $\mu$ is the normalization of
the Haar measure on $G$ that was used by Prasad in \cite{Pra89}, and that the 
volume is the half of this value otherwise. Suppose now that $\Gamma \subset G$ 
is torsion-free. Since $Z$ has order $2$ in all dimensions, we obtain in 
this case:
\begin{align}
  \vol(\Gamma \bs \Hy^n) &=  \vol(S^n) \mu(\Gamma \bs G)
  \label{eq:hyperb-vol-measure}
\end{align}
%The Lie group $G$ is of type $\tB_r$ with $r = n / 2$ if $n$ is odd. If
%$n$ is odd it has type $\tD_r$ with $r = (n+1)/ 2$. The center $Z(n)
%\subset G$ is then described according to $r$, and for its order we have
%\begin{align}
  %|Z(n)| &= \left\{
    %\begin{array}[h]{ll}
      %4 & \mbox{if } n \equiv 3 (4);\\
      %2 & \mbox{otherwise}.
    %\end{array} \right.
  %\label{eq:centers}
%\end{align}
If $n$ is even, then $2 \mu(\Gamma\bs G) = |\chi(\Gamma)|$
(cf. \cite[Section 4.2]{BorPra89}) and
equation \eqref{eq:hyperb-vol-measure} is
equivalent to the Gauss-Bonnet formula (note that a torsion-free
$\Gamma \subset G$ is isomorphic to -- and thus has same Euler characteristic as
-- its image in $\Isom^+(\Hy^n)$):
\begin{align}
  \vol(\Gamma\bs \Hy^n) &= \frac{\vol(S^n)}{2} |\chi(\Gamma)|,\ n \text{ even.}
  \label{eq:Gauss-Bonnet}
\end{align}
We recall that the volume of the $n$-sphere of curvature $1$ is given by the following formula:
\begin{align}
  \vol(S^n) = \frac{2 \pi^{\frac{n+1}{2}}}{\Gamma(\frac{n+1}{2})},
  \label{eq:vol-sphere}
\end{align}
where $\Gamma(-)$ is the Gamma function (recall that for an integer $m$ we
have $\Gamma(m) = (m-1)!$).

\subsection{}

Let $\Lambda_P \subset \Spin(n,1)$ be the (unique) nonuniform principal
arithmetic subgroup whose normalizer realizes the smallest covolume.
Like all nonuniform arithmetic hyperbolic lattices it is defined
over $k = \Q$, so that the set $\Vf$ of finite places corresponds to
the rational primes.  The structure of the coherent collection $P$ was
determined in \cite{Belo04,BelEme}. We denote by $\GG$ the algebraic
$\Q$-group containing $\Lambda_P$, and we will use the notation
introduced in Section~\ref{method-constr}.

\subsection{}

Let us first deal with  the case when $n = 2r$ is even. Then $\GG$ is of type $\tB_r$
and the coherent collection $P$
determining $\Lambda_P$ has the following structure (see \cite{Belo04}):
$P_v$ is hyperspecial
unless $v = 2$ and $r \equiv 2,3\; (4)$, in which case the associated
reductive $\F_2$-group $\MM_2$ is semisimple of type $^2\tD_r$. By
Prasad's volume formula~\cite{Pra89}, the Euler characteristic of $\Lambda_P$ is then
given by
\begin{align}
  |\chi(\Lambda_P)| &= 2 \lambda_2(r)
  C(r) \prod_{j=1}^r \zeta(2j),
  %\prod_{j=1}^r\frac{|B_{2j}|}{4j},
  \label{eq:EulerCh-noncompact-Cr}
\end{align}
with $\lambda_q(r) = 1$ if $r \equiv 0,1\;(4)$ and $\lambda_q(r) =
q^r - 1$ otherwise, and the constant $C(r)$ given by
\begin{align}
  C(r) &= \prod_{j=1}^r \frac{(2j-1)!}{(2\pi)^{2j}}.
  \label{eq:Cr}
\end{align}
Using the functional equation for the Riemann zeta function, and expressing its value at
negative integers through Bernoulli numbers, we can rewrite
\eqref{eq:EulerCh-noncompact-Cr} as follows:
\begin{align}
  |\chi(\Lambda_P)| &= 2 \lambda_2(r)
  \prod_{j=1}^r\frac{|B_{2j}|}{4j}.
  \label{eq:EulerCh-noncompact-Bern}
\end{align}
\subsection{}

We can apply the construction presented in Section \ref{method-constr}
to $\Lambda_P$ with $v = 2$ and $w = 3$. We find that for all $r$, we
have:
\begin{align}
  \lambda_2(r) [P_2: P_2^*] &= \prod_{j=1}^r (2^{2j} -1),
  \label{eq:coinc}
\end{align}
the case of trivial $\lambda_2(r)$ being just the computation in
\eqref{eq:index-Br}. Let $\Gamma$ be the (isomorphic) image of
$\Lambda_P^{2,3}$ in $\Isom^+(\Hy^n)$ and $M^n = \Gamma\bs \Hy^n$ the
corresponding quotient manifold, which by the construction is 
noncompact and arithmetic.
From the previous computation we have:
\begin{align}
  |\chi(M^{2r})| &= 2 \prod_{j=1}^r (4^j-1)(9^j-1) \frac{|B_{2j}|}{4j}
  \label{eq:noncompt-final}
  \\
  &= 2 C(r) \prod_{j=1}^r (4^j-1)(9^j-1) \zeta(2j)
  \label{eq:noncompt-final-Cr}.
\end{align}

\begin{table}
\def\arraystretch{1.1}
  \centering
  \begin{tabular}{r|r}
   \hline 
    $n$ & $|\chi(M^{2r})|$\\ \hline
   4 & 10\\
   6 & 910\\
   8 & 3171350\\
10 & 725639764850\\
12 & 16654568229539490250\\
14 & 54376724439679967985482572750\\
16 & 33998109351372684068956597092378802073750\\
18 & 5272397653068183031816584035192902513000228940543011250\\
%20 & 254994268996845725225855875260111112848653958492900597997560102795906250\\
   \hline
  \end{tabular}
  \vspace{0.5cm}
  \caption{Euler characteristic of noncompact manifolds $M^n$}
  \label{tab:Euler-char}
\end{table}

\subsection{}
We now construct noncompact hyperbolic manifolds in odd dimensions $n>3$.
Let $r = (n+1) / 2$. Then the $\Q$-group $\GG$ containing $\Lambda_P$
has type $\tD_r$. It is an inner form (type $^1\tD_r$) unless $r$ is
even (i.e., $n \equiv 3\; (4)$), in which case it has type $^2\tD_r$ and
becomes inner over $\ell = \Q(\sqrt{-3})$. The details of the construction
are given in \cite{BelEme}, here we only briefly recall the relevant
parts.

Let us first suppose that $r$ is odd. Then $P_v$ is hyperspecial unless
$v = 2$ and $r \equiv 3\;(4)$. In the latter case the radical of
$\MM_2$ is a nonsplit one-dimensional torus, and the semisimple part of
$\MM_2$ has type $^2\tD_{r-1}$. If we let
\begin{align}
  \lambda'_q(r) = \frac{(q^r-1)(q^{r-1}-1)}{q+1}
  \label{eq:lambda-odd}
\end{align}
for $q = 2$ with $r \equiv 3\; (4)$ and $\lambda'_q(r) = 1$ otherwise,
we find that in all cases with odd $r$ we have
\begin{align}
  \lambda'_{q_v}(r) [P_v:P_v^*]  &= (q_v^r-1) \prod_{j=1}^{r-1}(q_v^{2j}-1).
  \label{eq:index-r-odd}
\end{align}
Moreover, Prasad's formula together with
equations \eqref{eq:hyperb-vol-measure} and \eqref{eq:vol-sphere} shows
that we have (with $C(-)$ defined in \eqref{eq:Cr}):
\begin{align}
  \vol(\Lambda_P) &= \vol(S^n) \lambda'_2(r) \frac{(r-1)!}{(2\pi)^r} \zeta(r)
  \nonumber
  \prod_{j=1}^{r-1} \frac{(2j-1)!}{(2\pi)^{2j}} \zeta(2j)\\
  &= \lambda'_2(r) \frac{C(r-1)}{2^{r-1}} \zeta(r) \prod_{j=1}^{r-1} \zeta(2j)\\
  &=  \lambda'_2(r) \zeta(r) \prod_{j=1}^{r-1} \frac{|B_j|}{8j}.
  \label{eq:volO-r-3}
\end{align}
Thus in this case we obtain a manifold $M^n = \Lambda_P^{2,3}\bs \Hy^n$, $n=2r-1$, $r$ odd, whose
volume is given by
\begin{align}
  \vol(M^{n}) &= \zeta(r)(2^r-1)(3^r-1)
  \prod_{j=1}^{r-1}\frac{(4^j-1)(9^j-1)}{8j}|B_{2j}| .
  \label{eq-vol-manif-r-odd}
\end{align}

Let us finally consider the remaining case $n = 2r-1$ with $r$ even. In
this case all parahoric subgroups $P_v$ in the collection $P$ are
hyperspecial with the only exception of $v = 3$ (which ramifies in
$\ell = \Q(\sqrt{-3})$), where $P_v$ is special.
Since $2$ does not split in $\ell$, the group $\GG$ is
an outer form over $\Q_2$ and the $\F_2$-group $\MM_2$ associated to
the hyperspecial parahoric $P_2$ is semisimple of type $^2\tD_r$. Thus
we have:
\begin{align}
  [P_2:P_2^*] &= (2^r + 1) \prod_{j=1}^{r-1}(2^{2j}-1).
  \label{index-2-r-even}
\end{align}
For the place $v = 3$, $\MM_3$ is semisimple of type $\tB_{r-1}$, from
which we get
\begin{align}
  [P_3:P_3^*] &= \prod_{j=1}^{r-1} (3^{2j}-1).
  \label{index-3-r-even}
\end{align}
Using Prasad's formula to compute the covolume of $\Lambda_P$,  we
finally  deduce that the manifold $M^n = \Lambda_P^{2,3}\bs \Hy^n$, $n=2r-1$, $r$ even, has volume
\begin{align}
  \vol(M^{2r-1}) &= 3^{r-1 / 2} (2^r + 1) \zeta_\ell(r) / \zeta(r)
  \prod_{j=1}^{r-1} \frac{(4^j-1)(9^j-1)}{8j} |B_{2j}|.
  \label{eq-vol-manif-r-even}
\end{align}

\subsection{}

Formulas \eqref{eq-vol-manif-r-odd} and \eqref{eq-vol-manif-r-even} allow us to 
evaluate the volume of the constructed manifolds $M^n$ using
\textsc{Pari/GP}. Moreover, for $n$ even the volume can be deduced from
formula \eqref{eq:noncompt-final} together with \eqref{eq:Gauss-Bonnet}.
We list the values obtained for dimensions $n\le20$ in Table~\ref{tab:volumes-noncompact}.
It shows, in particular, that the volume grows rapidly with the dimension.
\begin{table}
  \def\arraystretch{1.1}
  \centering
  \begin{tabular}{r|r}
   \hline
   $n$ & $\vol(M^n)$ \\ \hline
   4 & 131.594 \\
   5 & 273.467 \\
   6 & 15048.379 \\
   7 & 42504.453 \\
   8 & 47073267.939 \\
   9 & 770938537.303 \\
   10 & 7519493827964.029 \\
   11 & 305396253769850.768 \\
   12 &  98579836734072034892.809 \\
   13 &  36424053767874477954431.531 \\
   14 & 1.555 E29 \\
   15 & 2.059 E32 \\
   16 & 4.074 E40 \\
   17 & 6.691 E44 \\
   18 & 2.335 E54 \\
   19 & 1.797 E59 \\
   20 & 3.734 E70 \\
   \hline
  \end{tabular}
  \vspace{0.5cm}
  \caption{Approximate values of the volume of $M^n$}
  \label{tab:volumes-noncompact}
\end{table}

\section{Proof of Theorem \ref{thm:conj-ok} for sufficiently large dimension}
\subsection{}

We now compare the volume of $M^n$ from the previous section with the volume of the smallest compact arithmetic
hyperbolic $n$-orbifold $O^n = \Gamma\bs \Hy^n$. In the rest of the
paper the Euler characteristic $\chi(O^n)$ is to be understood in the
orbifold sense, that is, $\chi(O^n) = \chi(\Gamma)$.

We begin with even dimensions $n = 2r \ge 4$. The orbifold $O^n$ is defined over the field $k = \Q(\sqrt{5})$
and its Euler characteristic is given by
\begin{align}
  |\chi(O^n)| = 4 \cdot \blambda_4(r)  \cdot 5^{r^2 + r/2} C(r)^2
  \prod_{j=1}^r \zeta_k(2j),
  \label{eq:Euler-cpt-orbi}
\end{align}
with $\blambda_q(r) = 1$ if $r$ is even and $\blambda_q(r) = \frac{q^r-1}{2}$ if $r$ is odd \cite{Belo04}.

By \eqref{eq:noncompt-final-Cr} and \eqref{eq:Euler-cpt-orbi}, we have:
\begin{align}  \label{eq:quotient-even1}
  \frac{|\chi(O^n)|}{|\chi(M^n)|} &= \frac{4 \cdot \blambda_4(r)  \cdot 5^{r^2 + r/2} C(r)^2 \prod_{j=1}^r \zeta_k(2j)}
{2 C(r) \prod_{j=1}^r (4^j-1)(9^j-1) \zeta(2j)}.
\end{align}
Using the basic inequalities $\zeta_k(2j) > 1$ and $\prod_{j=1}^r \zeta(2j) < 2$
plus the fact that $\blambda_4(r) \ge 1$ for all $r$, we obtain the lower bound
\begin{align}  \label{eq:quotient-even2}
  \frac{|\chi(O^n)|}{|\chi(M^n)|} &> \frac{5^{r^2 + r/2} C(r)}{\prod_{j=1}^r (4^j-1)(9^j-1)} > \frac{5^{r^2 + r/2} C(r)}{36^{r^2/2+r/2}}.
\end{align}
By Stirling's formula,
$$C(r) = \prod_{j=1}^r \frac{(2j-1)!}{(2\pi)^{2j}} > \prod_{j=1}^r \frac{(2j-1)^{2j-1}}{2\pi(2\pi e)^{2j-1}}.$$
Applying the Euler--Maclaurin summation formula, we obtain
$$\log C(r) \gtrsim \frac14(2r-1)^2\log(2r-1),$$
which is $\gg cr^2$ (for any constant $c$). Hence for $r\gg 0$, we have
\begin{align*}
  \frac{|\chi(O^n)|}{|\chi(M^n)|} &\gg 1.
\end{align*}
In fact, a computation shows that it is enough to take $r\ge 18$ for the ratio in the right hand side of \eqref{eq:quotient-even2} to be $>1$.
A more precise computation of the expression in \eqref{eq:quotient-even1} using \textsc{Pari/GP} shows that
\begin{align*}
  \frac{|\chi(O^n)|}{|\chi(M^n)|} &> 1, \text{ for } n = 2r \ge 34.
\end{align*}
Since $O^n$ has the smallest volume among compact arithmetic hyperbolic $n$-orbifolds,
this proves Theorem \ref{thm:conj-ok} for these values of $n$.

\subsection{}
Now consider the odd dimensional case $n = 2r-1 \ge 5$.

The smallest compact odd dimensional arithmetic orbifold $O^n$ is again defined over $k = \Q(\sqrt{5})$. Its volume is given by
\begin{align}
  \vol(O^n) &=
  \frac{5^{r^2-r/2} \cdot 11^{r-1/2} \cdot (r-1)!}{2^{2r-1} \pi^r} \cdot L_{\ell_0|k}\!(r) \cdot C(r-1)^2 \prod_{j=1}^{r-1} \zeta_{k}(2j),
  \label{eq-vol-cpt-orbi-n-odd}
\end{align}
where  $\ell_0$ is the quartic field with a defining polynomial
$x^4-x^3+2x-1$  and $L_{\ell_0|k} = \zeta_{\ell_0}/\zeta_k$ \cite[Thoerem~1]{BelEme}. 

Note that for any $s \ge 3$ we have $L_{\ell_0|k}(s) \ge 1/\zeta_k(3) > 0.973$. 

By \eqref{eq-vol-manif-r-odd}, \eqref{eq-vol-manif-r-even} and \eqref{eq-vol-cpt-orbi-n-odd}, we have
\begin{align}
  \frac{\vol(O^n)}{\vol(M^n)} &=
  \frac{
  \frac{5^{r^2-r/2}\cdot 11^{r-1/2} \cdot (r-1)!}{2^{2r-1} \pi^r} \cdot L_{\ell_0|k}\!(r) \cdot C(r-1)^2 \prod_{j=1}^{r-1} \zeta_{k}(2j)}
  {A(r) \cdot \frac{C(r-1)}{2^{r-1}} \cdot \prod_{j=1}^{r-1} \zeta(2j)(4^j-1)(9^j-1)},
  \label{eq:quotient-odd1}
\end{align}
where $A(r) = \zeta(r)(2^r-1)(3^r-1)$ if $r$ is odd and $A(r) = 3^{r-1/2}(2^r+1)\zeta_\ell(r)/\zeta(r)$, $\ell = \Q(\sqrt{-3})$, if $r$ is even, and hence
\begin{align}
  A(r) &< 6^r\cdot 2 \text{ for all } r.
\end{align}
Therefore,
\begin{align}
  \frac{\vol(O^n)}{\vol(M^n)} &>
  \frac{5^{r^2-r/2}\cdot 11^{r-1/2} \cdot (r-1)! \cdot C(r-1)}{2^{r} \pi^r 6^r\cdot 4 \cdot \prod_{j=1}^{r-1} (4^j-1)(9^j-1)} \\
  &> \frac{5^{(r-1)(r-1/2)}C(r-1)}{36^{r^2/2-r/2}} \cdot \frac{11^{r-1/2} \cdot (r-1)!}{5^{1/2-r} \cdot 4 \cdot (12\pi)^r}.
\end{align}
Now the first factor has the same form as the ratio in \eqref{eq:quotient-even2}, hence by the previous section it is $>1$ for $r-1 \ge 18$. It is an easy exercise to show that for such $r$, the second factor is $>1$ as well.

We can improve the bound for $r$ by evaluating \eqref{eq:quotient-odd1} using \textsc{Pari/GP}. This gives
\begin{align*}
  \frac{|\chi(O^n)|}{|\chi(M^n)|} &> 1, \text{ for } n = 2r-1 \ge 33.
\end{align*}
Together with the even dimensional part it proves Theorem \ref{thm:conj-ok} for $n \ge 33$.

\section{Lowering the dimension bound}

\subsection{}\label{sec:n=32}
Improving the bound for dimension requires more careful analysis. We will begin again with the even dimensional case.

The first dimension to consider is $n = 32$. Here we have a noncompact
arithmetic manifold $M^n$ with Euler characteristic $|\chi(M^n)| =
2.354\ldots\cdot10^{228}$. Ironically, we call it ``small'', but our
study shows that small manifolds in high dimensions tend to have huge
volume. The smallest compact arithmetic $n$-orbifold $O^n$ has
$|\chi(O^n)| = 8.777\ldots\cdot10^{217}$, which is less than
$|\chi(M^n)|$. Using \textsc{Pari/GP} we can compute (see remark below)
the precise value of $|\chi(O^n)|$, which is a rational number.  Its
denominator equals $$D=
107887196930872715055177987172922818560000000000000000000,$$ which is
$\sim 10^{56}$. A manifold covering of $O^n$ has to have an (even)
integer Euler characteristic, hence its degree is divisible by $D$. It
follows that the volume of the smallest possible manifold cover of $O^n$
is much larger than $\vol(M^n)$.

\begin{rmk}
 A way to compute the exact value of $|\chi(O^n)|$ is to first use the
 functional equation for $\zeta_k$ to  transform
 equation \eqref{eq:Euler-cpt-orbi} into a product of a (known) rational factor 
 and values $\zeta_k(1-2j)$, which are rational. Using \textsc{Pari/GP}
 the numerical values of $\zeta_k(1-2j)$ can be approximated by rationals,
 leading to the number $D$
 (it is enough to work with precision \verb?\p 80?). 
 With this method based on approximation, the value obtained for $D$ can
 only be used as a lower bound for the actual denominator of
 $|\chi(O^n)|$, but this is already sufficient for our purpose. However, a
 cleaner way to proceed is to use the library of special values
 computed by Alvaro Lozano-Robledo, who obtained them through a procedure
 under \textsc{Pari/GP} that computes generalized Bernoulli 
 numbers (see \cite{LozRob}). This guarantees the correctness of the value computed for $D$.
\end{rmk}

In order to complete the discussion for $n=32$ we need to consider other
maximal arithmetic subgroups $\Gamma$ whose covolume is $< \vol(M^n)$.
Recall that any maximal arithmetic subgroup is a normalizer of a
principal subgroup.
Maximal arithmetic subgroups other than the one defining $O^n$ can be either
subgroups defined over $k = \Q(\sqrt{5})$ 
or arithmetic subgroups defined over other totally real fields. 
In the latter case, by \cite{Belo04}, we know that the next
smallest covolume subgroup has the field of definition $\Q(\sqrt{2})$.
By the volume formula, the absolute value of its Euler characteristic is
$>9.071\cdot10^{271}$, which is already bigger than $|\chi(M^n)|$.
Therefore, we are only left with the groups defined over $k$. The
covolume of any such subgroup $\Gamma$ would have an additional lambda factor in
comparison with $\vol(O^n)$. By
\cite[Sections 2.3 and 3.2]{Belo04}, the smallest possible lambda factor 
is $\lambda_v = \frac{q_v^{16}+1}{2}$ for $q_v = 4$,
and it follows that the orbifold corresponding to the principal
arithmetic sugroup $\Lambda$ that $\Gamma$ normalizes has 
an Euler characteristic  of absolute value
$1.884\ldots\cdot 10^{227}$ with denominator $2D$.
Hence all manifold covers of this orbifold are much larger than $M^n$. Now,
the index $[\Gamma : \Lambda]$ can be computed  the same way as it is done in
\cite{Belo04} and \cite{BelEme} and in particular, since the field
$k = \Q(\sqrt{5})$
has class number one, this index is a power of $2$ (whose exponent can
be bounded). Using this the result just obtained immediately extends to 
the manifold covers of $\Gamma\bs \Hy^{32}$.
The same argument excludes the case of orbifolds with a lambda factor
$\lambda_v$ corresponding to $q_v = 5$. 
All other possible lambda factors give rise to orbifolds of volume
larger than $\vol(M^n)$. This shows that Theorem \ref{thm:conj-ok} is
true for $n = 32$.

%\begin{rmk} Strictly speaking, the considerations in the last paragraph
  %had to take into account the fact that maximal arithmetic lattices
  %$\Gamma_m$ may strictly contain principal lattices $\Gamma$ whose
  %volume we control there, meaning that we also need to control the
  %index $[\Gamma_m:\Gamma]$ and enlarge the range of volumes of
  %potential candidates by the upper bound for this index. In our
  %computation here and in the following sections we take advantage of
  %the fact that all the admissible groups $\G$ have the same structure
  %as the groups giving rise to the minimal covolume arithmetic subgroups
  %and hence the related index can be computed the same way as in
  %\cite{Belo04} and \cite{BelEme}. We will not repeat the details of the
  %related estimation and computation here. In smaller dimensions the
  %range of admissible groups is larger and one has to take much more
  %care about this index.  \end{rmk}

\subsection{}\label{sec:n=30} 
Dimension $n = 30$ is treated in a similar
way. Here we have a noncompact arithmetic manifold $M^n$ with
$|\chi(M^n)| = 1.252\ldots\cdot10^{195}$. The smallest volume compact
arithmetic orbifold $O^n$ has $|\chi(O^n)| = 8.112\ldots\cdot10^{187}$
with denominator 
$~5.231\ldots\cdot10^{48}$,
%$~4.243\ldots\cdot10^{236} > |\chi(M^n)|$,
hence its
manifold covers are larger than $M^n$. The smallest covolume cocompact
arithmetic subgroup defined over $k \neq \Q(\sqrt{5})$ has the field of
definition $\Q(\sqrt{2})$ and $|\chi| > 3.116\cdot10^{235}$, which is
bigger than $|\chi(M^n)|$. For the other maximal arithmetic subgroups
defined over $\Q(\sqrt{5})$ we have a different lambda factor in the
volume formula. The values of $\lambda$ for which $|\chi|$ of the
principal subgroup is smaller
than $|\chi(M^n)|$ are listed below: 
$$ \def\arraystretch{1.3}
\begin{array}{c|c|c}
\hline 
\lambda & |\chi| & \text{denominator of }\chi \\
\hline 
\frac12(5^{15}-1) & 2.305...\cdot10^{189} & 82391859826240770906019357261824\cdot10^{18}\\ 
\frac12(9^{15}-1) & 1.555...\cdot10^{193} & 41154775137982403049959718912\cdot10^{18}\\
\frac12(11^{15}-1) & 3.155...\cdot10^{194} & 32109284800854974718802972372893696\cdot10^{16}\\ 
\hline
\end{array} $$

As before, considering maximal sugroups instead of principal subgroups
does not change the picture and it follows that a manifold
cover in each of these cases is larger than the noncompact manifold $M^n$.

\subsection{}\label{sec:odd-dim-trick}
Odd dimensional case presents us a different challenge: here the Euler characteristic is zero and we cannot take advantage of its integral properties in order to bound the degree of the smooth covers. One of the possible ways to proceed is to look at the orders of finite subgroups of $\pi_1(O^n)$. This indeed gives a bound for the degree, however, it is much smaller than bounds provided by the denominators of Euler characteristic in the neighboring even dimensions. Considerably stronger results can be obtained based on the following simple observation:

\begin{quotation}
  Small volume arithmetic hyperbolic orbifolds tend to contain totally geodesic even dimensional suborbifolds whose Euler characteristic can be used to obtain good bounds on the degrees of the smooth covers.
\end{quotation}

Indeed, assume that a group $\Gamma$ has an infinite index subgroup $\Gamma'$ and a torsion-free finite index subgroup $\Gamma_M$ with $[\Gamma:\Gamma_M] = f$. Then $\Gamma_M \cap \Gamma'$ is a torsion-free subgroup of $\Gamma'$ of index $d \le f$:
\begin{equation}
	\xymatrix{
	\Gamma \ar@{-}[r]^{\infty} \ar@{-}[d]^{f} & \Gamma' \ar@{-}[d]^{d} \\
	\Gamma_M    \ar@{-}[r]^-\infty & \Gamma_M\cap\Gamma'
	}
	\label{eq-degree-bound}
\end{equation}
This proposition can be easily checked by looking at the cosets $\Gamma/\Gamma_M$ and $\Gamma'/(\Gamma_M\cap\Gamma')$. It appears to be very useful for bounding degrees of the smooth covers in odd dimensions.

\subsection{}\label{sec:n=31}
We demonstrate the application of the method from \ref{sec:odd-dim-trick} for dimension $n = 31$. By
\cite[Section~3.5]{BelEme}, the minimal volume compact orbifold $O^{31}$ is defined by the quadratic form
$$f_{31}(x_0,x_1,\ldots,x_{31}) = (3-2\sqrt{5})x_0^2 + x_1^2 + \ldots x_{31}^2.$$ 
It has $\vol(O^{31}) = 2.415\ldots\cdot 10^{200} < \vol(M^{31}) = 3.113\ldots\cdot 10^{202}$, 
so we can not a priori exclude the possibility of some manifold cover of $O^{31}$ being smaller than $M^{31}$. 

A restriction $f_{30}$ of the form $f_{31}$ to the first $31$ coordinates defines a totally geodesic 
suborbifold $O_1^{30}\subset O^{31}$ that has a non-zero Euler characteristic. In order to 
compute $\chi(O_1^{30})$ we need to control the associated integral structure.

In what follows we will use the notation from \cite{BelEme} and \cite{GanHankeYu01}. Consider the quadratic space $(V, \frac12 f_{31})$.
We have the discriminant $\delta(\frac12f_{31}) = \frac{1}{2^{32}}(3-2\sqrt{5})$
and the Hasse-Witt invariant $\omega(V,\frac12f_{31}) = 1$ over all places including the dyadic place $v_2 = (2)$ of $k = \Q(\sqrt{5})$. 
Comparing this data with the local description of the group $\Gamma$ of $O^{31}$ given in \cite[Section~3]{BelEme}, we conclude that 
$\Gamma$ is isomorphic to the stabilizer of a maximal lattice $L$ in $(V,\frac12f_{31})$  
(this can be also confirmed by comparing the covolume of $\mathrm{Stab}(L)$ computed using \cite[Table~3]{GanHankeYu01} with $\vol(O^{31})$
and recalling the uniqueness of $O^{31}$).
The lattice $L$ is defined uniquely up to a conjugation under $\G$ but different (conjugate) lattices may have different 
restriction to the subspace $(V', \frac12f_{30})$ of $V$. Geometrically this corresponds to choosing different totally
geodesic $30$-dimensional suborbifolds of $O^{31}$. In order to complete the computation we need to fix the lattice $L$. 
This can be done as follows: First take a maximal lattice $L'$ in $(V' , \frac12f_{30})$. Now consider an integral lattice in $V$
generated by $L'$ and the vector $(0, ..., 0, 2)$. It is contained in some maximal lattice in $V$ and we can choose this lattice to be our 
lattice $L$. This construction implies that the restriction of $L$ onto the subspace $V'\subset V$ is the maximal lattice $L'$ in 
$(V' , \frac12f_{30})$. Now we can use \cite[Table~4]{GanHankeYu01} to compute the covolume of $\mathrm{Stab}(L')$, and hence the 
Euler characteristic of the corresponding suborbifold:
$$|\chi(O_1^{30})| = \frac{4^{15}-1}2 \cdot \frac{11^{15}+1}2 \cdot \frac{1}{4^{14}}\prod_{i=1}^{15} |\zeta_k(1-2i)| = 1.694\ldots\cdot 10^{202},$$ 
with denominator 
$$D_1 = 290623844184270796846629126144000000000000000000.$$ 
Hence by \eqref{eq-degree-bound}, the degree $f$ of any smooth cover of $O^{31}$ is $\ge D_1$, and so its volume 
\begin{align} \vol(O^{31})\cdot f > 7.019 \ldots\cdot 10^{247} > \vol(M^{31}).  \end{align}

It remains to check if there are other compact arithmetic $31$-dimensional orbifolds that
have volume $\le \vol(M^{31})$. Similar to the previous discussion, these can be either defined over a different
field $k$, or have a different splitting field $\ell$, or have a non-trivial lambda factor in the volume formula. 
Note that by \eqref{eq:quotient-odd1} the ratio
\begin{align}
  \frac{\vol(O^{31})}{\vol(M^{31})} &> 0.007.
\end{align}
If we change the defining field then its discriminant in the volume formula for
$\vol(O^{31})$ would contribute a factor of at least $(8/5)^{r^2-r/2}
(1/11)^{r-1/2}> 3.021\cdot 10^{34}$ (cf.~\cite{BelEme}). This immediately brings it out of the range of
consideration. Changing the field $\ell$ gives a factor $\ge \left(\frac{400}{275}\right)^{31/2} = 332.868\ldots$
(as the field $\ell$ that corresponds to the minimal volume orbifold has discriminant
$\mathcal{D}_\ell = 275$ and the next possible value is $\mathcal{D} = 400$). This factor
is already sufficiently large to make the volume $>\vol(M^{31})$. Finally, following 
\cite[equation~ (5)]{BelEme}, the lambda factor is bounded below by 
$\frac23\left(\frac34 q_v\right)^{r_v}$ with $q_v \ge 4$, $r_v = 16$ if $q_v\neq 11$ and $r_v = 15$ for the remaining place $v$ of $k$ that ramifies in $\ell/k$. 
Hence the smallest possible lambda factor is at least $28697814$, which again makes the volume of the corresponding orbifolds 
too large. Similar to the previous sections, this argument based on principal arithmetic subgroups extends to maximal arithmetics 
in a straightforward way. 

This finishes the proof of Theorem \ref{thm:conj-ok}. \qed

\bibliographystyle{amsplain}
\bibliography{manifolds}

\providecommand{\bysame}{\leavevmode\hbox to3em{\hrulefill}\thinspace}
\providecommand{\MR}{\relax\ifhmode\unskip\space\fi MR }
\providecommand{\MRhref}[2]{%
  \href{http://www.ams.org/mathscinet-getitem?mr=#1}{#2}
}
\providecommand{\href}[2]{#2}
\begin{thebibliography}{10}

\bibitem{Anderson06}
Michael~T. Anderson, \emph{Dehn filling and {E}instein metrics in higher
  dimensions}, J. Differential Geom. \textbf{73} (2006), no.~2, 219--261.

\bibitem{Belo04}
Mikhail Belolipetsky, \emph{On volumes of arithmetic quotients of
  {$\mathrm{SO}(1,n)$}}, Ann. Sc. Norm. Super. Pisa Cl. Sci. (5) \textbf{3}
  (2004), no.~4, 749--770, Addendum: ibid., 6 (2007), 263--268.

\bibitem{BelEme}
Mikhail Belolipetsky and Vincent Emery, \emph{On volumes of arithmetic
  quotients of {$\mathrm{PO}(n,1)^\circ$}, {$n$} odd}, Proc. Lond. Math. Soc.
  (3) \textbf{105} (2012), no.~3, 541--570.

\bibitem{BorPra89}
Armand Borel and Gopal Prasad, \emph{Finiteness theorems for discrete subgroups
  of bounded covolume in semi-simple groups}, Inst. Hautes {\'E}tudes Sci.
  Publ. Math. \textbf{69} (1989), 119--171, Addendum: ibid., 71 (1990),
  173--177.

\bibitem{ERT10}
Brent Everitt, John~G. Ratcliffe, and Steven~T. Tschantz, \emph{Right-angled
  {C}oxeter polytopes, hyperbolic six-manifolds, and a problem of {S}iegel},
  Math. Ann. \textbf{354} (2012), no.~3, 871--905.

\bibitem{GMM10}
David Gabai, Robert Meyerhoff, and Peter Milley, \emph{Minimum volume cusped
  hyperbolic three-manifolds}, J. Amer. Math. Soc. \textbf{22} (2009), no.~4,
  1157--1215.

\bibitem{GanHankeYu01}
Wee~Teck Gan, Jonathan~P. Hanke, and Jiu-Kang Yu, \emph{On an exact mass
  formula of {S}himura}, Duke Math. J. \textbf{107} (2001), no.~1, 103--133.
  \MR{1815252 (2002c:11044)}

\bibitem{Hild07}
Thierry Hild, \emph{The cusped hyperbolic orbifolds of minimal volume in
  dimensions less than ten}, Journal of Algebra \textbf{313} (2007), no.~1,
  208--222.

\bibitem{LozRob}
{\'A}lvaro Lozano-Robledo, \emph{Values of {D}edekind zeta functions at
  negative integers},
  \url{http://homepage.uconn.edu/alozano/dedekind-values/index.html}.

\bibitem{Ono66}
Takashi Ono, \emph{On algebraic groups and discontinuous groups}, Nagoya Math.
  J. \textbf{27} (1966), 279--322.

\bibitem{PlaRap94}
Vladimir Platonov and Andrei~S. Rapinchuck, \emph{Algebraic groups and number
  theory (engl. transl.)}, Pure and applied mathematics, vol. 139, Academic
  Press, 1994.

\bibitem{Pra89}
Gopal Prasad, \emph{Volumes of {$S$}-arithmetic quotients of semi-simple
  groups}, Inst. Hautes {\'E}tudes Sci. Publ. Math. \textbf{69} (1989),
  91--117.

\bibitem{Ratcliffe2001899}
John~G. Ratcliffe, \emph{Chapter 17 - hyperbolic manifolds}, Handbook of
  Geometric Topology (R.J. Daverman and R.B. Sher, eds.), North-Holland,
  Amsterdam, 2001, pp.~899--920.

\bibitem{RatTsch00}
John~G. Ratcliffe and Steven~T. Tschantz, \emph{The volume spectrum of
  hyperbolic 4-manifolds}, Experiment. Math. \textbf{9} (2000), no.~1,
  101--125. \MR{1758804 (2001b:57048)}

\bibitem{Thur80}
William~P. Thurston, \emph{The geometry and topology of 3-manifolds}, 1979,
  Princeton Univ., online version at:
  \url{http://www.msri.org/publications/books/gt3m/}.

\bibitem{Tits79}
Jacques Tits, \emph{Reductive groups over local fields}, Proc. Sympos. Pure
  Math., vol.~33, 1979, pp.~29--69.

\end{thebibliography}

\end{document}